\documentclass[12pt]{amsart}
\usepackage{amsfonts}
\usepackage{amsmath}
\usepackage{amssymb}
\usepackage[margin=1in]{geometry}
\newtheorem{theorem}{Theorem}

\newtheorem{lemma}{Lemma}

\theoremstyle{definition}

\begin{document}
\title[The absence of remainders in the Wiener-Ikehara theorem]{Note on the absence of remainders in the Wiener-Ikehara theorem}

\author[G. Debruyne]{Gregory Debruyne}
\thanks{G. Debruyne gratefully acknowledges support by Ghent University, through a BOF Ph.D. grant.}
\address{Department of Mathematics\\ Ghent University\\ Krijgslaan 281\\ B 9000 Gent\\ Belgium}
\email{gregory.debruyne@UGent.be}
\author[J. Vindas]{Jasson Vindas}
\thanks{The work of J. Vindas was supported by Ghent University through
the BOF-grant 01J11615 and by the Research Foundation--Flanders through the FWO-grant 1520515N}
\address{Department of Mathematics\\ Ghent University\\ Krijgslaan 281\\ B 9000 Gent\\ Belgium}
\email{jasson.vindas@UGent.be}

\begin{abstract} 
We show that it is impossible to get a better remainder than the classical one in the Wiener-Ikehara theorem even if one assumes analytic continuation of the Mellin transform after subtraction of the pole to a half-plane. We also prove a similar result for the Ingham-Karamata theorem.
\end{abstract}

\maketitle

\section{Introduction}

The Wiener-Ikehara theorem is a landmark in 20th century analysis. It states

\begin{theorem} Let $S$ be a non-decreasing function and suppose that
\begin{equation} \label{eqmellin abserrorW-I}
 G(s) := \int^{\infty}_{1} S(x) x^{-s-1} \mathrm{d}x \text{ converges for } \Re e \: s > 1
\end{equation}
and that there exists $a$ such that $G(s) - a/(s-1)$ admits a continuous extension to $\Re e \: s \geq 1$, then 
\begin{equation}
\label{eq conclusion abserrorW-I}
S(x) = ax + o(x).
\end{equation}
\end{theorem}
This result is well-known in number theory as it leads to one of the quickest proofs of the prime number theorem. However, it has also important applications in other fields such as operator theory (see e.g. \cite{aramaki}). Over the last century the Wiener-Ikehara theorem has been extensively studied and generalized in many ways (e.g., \cite{ d-vW-I2016, d-vCT, delange1954, grahamvaaler, korevaar2006, revesz-roton, Tenenbaumbook,zhang2014}). We refer the interested reader to \cite[Chap. III]{korevaarbook} for more information about the Wiener-Ikehara theorem.\par

If one wishes to attain a stronger remainder in \eqref{eq conclusion abserrorW-I} (compared to $o(x)$), it is natural to strengthen the assumptions on the Mellin transform \eqref{eqmellin abserrorW-I}. We investigate here whether one can obtain remainders if the Mellin transform after subtraction of the pole at $s=1$ admits an analytic extension to a half-plane $\Re e \:s > \alpha$ where $0 < \alpha < 1$. It is well-known that one can get reasonable error terms in the asymptotic formula for $S$ if bounds are known on the analytic function $G$. The question of obtaining remainders if one does not have such bounds was recently raised by M\"uger \cite{muger}, who actually  conjectured the error term $O(x^{(\alpha + 2)/3 + \varepsilon})$ could be obtained for each $\varepsilon>0$. 

We show in this article that this is false. In fact, we shall prove 
in Section \ref{Section 3 abserrorW-I} 
the more general result that no reasonably good remainder can be expected in the Wiener-Ikehara theorem, with solely the classical Tauberian condition (of $S$ being non-decreasing) and the analyticity of $G(s) - A/(s-1)$ on  $\Re e \:s > \alpha$ for $0<\alpha<1$.
To show this result we will adapt an attractive functional analysis argument given by Ganelius\footnote{According to him \cite[p. 3]{ganelius}, the use of functional analysis arguments to avoid cumbersome constructions of counterexamples in Tauberian theory was suggested by L. H\"{o}rmander.} \cite[Thm. 3.2.2]{ganelius}. Interestingly, the nature of our problem requires to consider a suitable Fr\'{e}chet space of functions instead of working with a Banach space.

In Section \ref{ingham section abserrorterm} we shall apply our result on the Wiener-Ikehara theorem to study another cornerstone in complex Tauberian theory, namely, the Ingham-Karamata theorem for Laplace transforms \cite[Chap. III]{korevaarbook} (see \cite{d-vOptIngham,d-vCT} for sharp versions of it). Notably, a very particular case of this theorem captured special attention when Newman found an elementary contour integration proof that leads to a simple deduction of the prime number theorem; in fact, this proof is nowadays a chapter in various popular expository textbooks in analysis \cite{C-Qbook,L-Zbook}. We will show  that, just as for the Wiener-Ikehara theorem, no reasonable error term can be obtained  in the Ingham-Karamata theorem under just an analytic continuation hypothesis on the Laplace transform. On the other hand, the situation is then pretty much the same as for the Wiener-Ikehara theorem, error terms can be achieved if the Laplace transform satisfies suitable growth assumptions. We point out that the problem of determining such growth conditions on the Laplace transform has been extensively studied in recent times \cite{B-B-T2016,Chill-Seifert2016,Seifert2015} and such results have numerous applications in operator theory and in the study of the asymptotic behavior  of solutions to various evolution equations.

\section{Some lemmas}
 We start with some preparatory lemmas that play a role in our constructions. The first one is a variant of the so-called smooth variation theorem from the theory of regularly varying functions \cite[Thm. 1.8.2, p. 45]{binghambook}. 

\begin{lemma} \label{lemregularization abserrorW-I} Let  $\ell$ be a positive non-increasing function on $[0,\infty)$ such that $\ell(x)=o(1)$ (as $x\to\infty$). Then, there is a positive smooth function $L$ such that 
$$
\ell(x)\ll L(x) = o(1),$$ 
and, for some positive $C$,$A$ and $B$,
\begin{equation} \label{eqreg abserrorW-I}
 \left|L^{(n)}(x)\right| \leq CA^{n}n!x^{-n},  \qquad \text{for all } x \geq B \mbox{ and }  n\in\mathbb{N}. 
\end{equation}
\end{lemma}
\begin{proof} We consider the Poisson kernel of the real line
$$
P(x,y) =\frac{y} {\pi(y^2+x^{2})}=\frac{i}{2\pi}\left(\frac{1}{x+iy}-\frac{1}{x-iy}\right).
$$
Differentiating the last expression with respect to $y$, it is clear that we find
\begin{align*}
 \left|\frac{\partial^{n}P}{\partial y}(x,y)\right|& \leq \frac{2^{n+1}n! y^{n+1}}{\pi (y^2+x^{2})^{1+n}} \max_{0\leq j\leq n+1 } |x/y|^{j}
 \\
 &
<  \frac{2^{n+1}n!}{\pi (y^2+x^{2})^{(1+n)/2}}, \quad \text{for all }n\geq1.
\end{align*}
We set 
$$L(y) = \int^{\infty}_{0} \ell(x y) P(x,1) \mathrm{d}x= \int^{\infty}_{0} \ell(x) P(x,y) \mathrm{d}x.
$$ By the dominated convergence theorem, we have $L(y)= o(1)$. Since $\ell$ is non-increasing and $P(x,1)$ positive, it follows that 
$$L(y)\geq \int^{1}_{0} \ell(y) P(x,1) \mathrm{d}x= \frac{\ell(y)}{4}\: .$$
For the derivatives we have
$
\left|L^{(n)}(y)\right| \leq \ell(0) 2^{n} n! y^{-n}$ for all  $n\in\mathbb{N}$ and  $y>0$.

\end{proof}

We also need to study the analytic continuation of the Laplace transform of functions satisfying the regularity assumption (\ref{eqreg abserrorW-I}).

\begin{lemma} \label{lemanalyticextension} Suppose that $L\in L^{1}_{loc}[0,\infty)$ satisfies the regularity assumption \eqref{eqreg abserrorW-I}
for some $A,B,C>0$ and set $\theta=\arccos(1/(1+A))$. Then its Laplace transform $\mathcal{L}\{L;s\} = \int^{\infty}_{0}e^{-sx}L(x)\mathrm{d}x$ converges for $\Re e \: s > 0$ and admits analytic continuation to the sector $-\pi +\theta<\arg s< \pi-\theta$.
\end{lemma}
\begin{proof} It is clear that $F(s)=\int^{\infty}_{0}e^{-sx}L(x)\mathrm{d}x$ converges for $\Re e \: s > 0$. Since the Laplace transform of a compactly supported function is entire, we may suppose that $L$ is supported on $[B,\infty)$. Since we can write $F(s) = e^{-sB}\int^{\infty}_{0} e^{-sx}L(x+B)\mathrm{d}x$, we may  w.l.o.g. assume $B = 0$ and replace $x^{-n}$ in the estimates for $L^{(n)}(x)$ by $(1+x)^{-n}$. We consider the $k$th derivative of $F$, namely $(-1)^{k}\int^{\infty}_{0} x^{k}e^{-sx}L(x)\mathrm{d}x$. We use integration by parts $k+2$ times to find
\begin{equation*}
 F^{(k)}(s) = (-1)^{k}\frac{k!L(0)}{s^{k+1}} + (-1)^{k}\frac{(k+1)!L'(0)}{s^{k+2}} + \frac{(-1)^{k}}{s^{k+2}} \int^{\infty}_{0} (L(x)x^{k})^{(k+2)}e^{-sx}\mathrm{d}x.
\end{equation*}
Because of the regularity assumption \eqref{eqreg abserrorW-I} the latter integral absolutely converges and hence $F$ admits a $C^{\infty}$-extension on the imaginary axis except possibly at the origin. The bounds \eqref{eqreg abserrorW-I} actually give for arbitrary $\varepsilon > 0$
\begin{align*}
 \left|F^{(k)}(it)\right| & \leq \frac{|L(0)|}{\left|t\right|} \frac{k!}{\left|t\right|^{k}} + \frac{|L'(0)|}{\left|t\right|^{2}} \frac{(k+1)!}{\left|t\right|^{k}} 
 +  \frac{1}{\left|t\right|^{k+2}} \int^{\infty}_{0} \sum_{j=2}^{k+2}{k+2 \choose j}  \left|L^{(j)}(x)\right|\frac{k!}{(j-2)!}x^{j-2}\mathrm{d}x
 \\
& \leq C' \frac{(1+|t|)(k+1)!}{\left|t\right|^{k+2}} + \frac{1}{\left|t\right|^{k+2}}\int^{\infty}_{0} \sum_{j = 2}^{k+2} CA^{j}\frac{(k+2)!k!}{(k+2-j)!(j-2)!}(1+x)^{-2} \mathrm{d}x
\\
& 
\leq C'\frac{(1+|t|)(k+1)!}{\left|t\right|^{k+2}} + \frac{A^2C\pi(k+2)!}{2\left|t\right|^{k+2}}\sum_{j = 0}^{k} A^{j} {k \choose j} \leq C_{\varepsilon}\frac{(1+|t|)}{|t|^{2}} \frac{k!(1+A+\varepsilon)^{k}}{\left|t\right|^{k}},
\end{align*}
where $C_{\varepsilon}$ only depends on $\varepsilon$ and $L$. Therefore, $F$ admits an analytic extension to the disk around $it$ with radius $\left|t\right|/(1+A)$. The union of all such disks and and the half-plane $\Re e\: s>0$ is precisely the sector in the statement of the lemma.
\end{proof}

\section{Absence of remainders in the Wiener-Ikehara theorem}
\label{Section 3 abserrorW-I}
 We are ready to show our main theorem, which basically tells us that no remainder of the form $O(x\rho(x))$ with $\rho(x)$ a function tending arbitrarily slowly to 0 could be expected in the Wiener-Ikehara theorem from just the hypothesis of analytic continuation of $G(s)-a/(s-1)$ to a half-plane containing $\Re e\: s\geq 1$. As customary, the $\Omega$ below stands for the Hardy-Littlewood symbol, namely, the negation of Landau's $o$ symbol. Our general reference for functional analysis is the text book \cite{treves}.
\begin{theorem}\label{th2 abserrorW-I} Let $\rho$ be a positive function,  let  $a> 0$, and $0 < \alpha < 1$. Suppose that every non-decreasing function $S$ on $[1,\infty)$, whose Mellin transform $G(s)$ is such that $G(s) - a/(s-1)$ admits an analytic extension to $\Re e \: s > \alpha$, satisfies 
$$S(x) = ax + O(x\rho(x)).$$
Then, one must necessarily have 
$$
\rho(x)=\Omega(1).
$$
\end{theorem}
\begin{proof} Since $a> 0$, we may actually assume that the ``Tauberian theorem'' hypothesis holds for every possible constant $a> 0$.  Assume that $\rho(x) \rightarrow 0$. Then, one can choose a positive non-increasing function $\ell (x)\rightarrow 0$ such that $\ell(\log x)/\rho(x) \rightarrow \infty$. We now apply Lemma \ref{lemregularization abserrorW-I} to $\ell$ to get a smooth function $L$ with $\ell(x)\ll L(x) \rightarrow 0$ and the estimates (\ref{eqreg abserrorW-I}) on its derivatives.  We set $x\rho(x)=1/\delta(x)$. If we manage to show that $\delta(x)=O(1/xL(\log x))$, then one obtains a contradiction with $\ell(\log x)/\rho(x) \rightarrow \infty$ and hence $\rho(x) \nrightarrow 0$. We thus proceed to show that $\delta(x)=O(1/xL(\log x))$. Obviously, we may additionally assume that $L$ satisfies 
\begin{equation}
\label{abserrorW-I eq 3} L(x)\gg x^{-1/2}.
\end{equation}

 We are going to define two Fr\'echet spaces. The first one consists of all Lipschitz continuous functions on $[1,\infty)$ such that their Mellin transforms can be analytically continued to $\Re e\: s> \alpha$ and continuously extended to the closed half-plane $\Re e\: s\geq \alpha$. We topologize it via the countable family of norms \[
 \left\|T\right\|_{n,1} = \operatorname*{ess\:sup}_{x}|T'(x)| + \sup_{\Re e \: s \geq \alpha,\: \left|\Im m \: s\right| \leq n} \left|G_{T}(s)\right|,
\]
where $G_{T}$ stands for (the analytic continuation of) the Mellin transform of $T$. The second Fr\'echet space is defined via the norms
\[
 \left\|T\right\|_{n,2} = \sup_{x} \left|T(x)\delta(x)\right| + \left\|T\right\|_{n,1}.
\]
The hypothesis in the theorem ensures that the two spaces have the same elements. The verification of completeness with respect to this family of norms is standard and we therefore omit it. Obviously the inclusion mapping from the second space into the first one is continuous. Hence, by the open mapping theorem, the inclusion mapping from the first space into the second one is also continuous. Therefore, there exist sufficiently large $N$ and $C$ such that
\begin{equation}
\label{eq 4 abserrorW-I}
  \sup_{x} \left|T(x)\delta(x)\right| \leq C \left\|T\right\|_{N,1}
\end{equation}
for all $T$ in our Fr\'{e}chet space. This inequality extends to the completion of the Fr\'echet space with regard to the norm $\left\|\:\cdot\:\right\|_{N,1}$. We note that any function $T$ for which $T'(x) = o(1)$, $T(1)=0$,  and whose Mellin transform has analytic continuation in a neighborhood of $\{s : \Re e \: s \geq \alpha, \left|\Im m \: s\right| \leq N\}$ is in that completion. Indeed, let $\varphi \in \mathcal{S}(\mathbb{R})$ be such that $\varphi(0)=1$ and its Fourier transform has compact support; then $\tilde{T}_{\lambda}(x) := \int^{x}_{1} T'(u) \varphi(\lambda \log u)\mathrm{d}u$ converges to $T$ as $\lambda \rightarrow 0^{+}$ in the norm $\left\|\:\cdot\:\right\|_{N,1}$. We now consider
\[
 T_{b}(x) := \int^{x}_{1} L(\log u) \cos(b \log u) \mathrm{d}u.
\]
Obviously the best Lipschitz constant for $T_{b}$ is bounded by the supremum of $L$. Its Mellin transform is 
\[
 G_{b}(s) = \frac{1}{2s}\left(\mathcal{L}\{L;s-1+ib\} + \mathcal{L}\{L;s-1-ib\}\right).
\]
Because of Lemma \ref{lemanalyticextension} it follows that $G_{b}$ is analytic in $\{s : \Re e \: s \geq \alpha, \left|\Im m \: s\right| \leq N\}$ for all sufficiently large $b$, let us say for every $b > M$. Hence, the norm $\left\|T_{b}\right\|_{N,1}$ is uniformly bounded in $b$ for $b \in [M,M+1]$. A quick calculation shows for $b \in [M,M+1]$
\[
 T_{b}(x) :=  \frac{x L(\log x)}{b^{2}+ 1} \left(\cos(b\log x)+ b\sin(b\log x)\right)+ O\left(\frac{x}{\log x}\right),
\]
where the $O$-constant is independent of $b$. For each $y$ large enough there is $b \in [M,M+1]$ such that $\sin(b\log y) = 1$. Therefore, for $y$ sufficiently large, taking also \eqref{abserrorW-I eq 3} into account, we have 
\[
 \sup_{b \in [M,M+1]} T_{b}(y) \geq \inf_{b \in [M,M+1]} \frac{byL(\log y)}{b^{2} + 1} +O\left(\frac{y}{\log y}\right) \geq C_{M}yL(\log y),
\]
with $C_{M}$ a positive constant. Consequently, for all sufficiently large $y$, the inequality \eqref{eq 4 abserrorW-I} yields

\begin{align*}
 \delta(y)&  \leq \sup_{b \in [M,M+1]} \frac{ T_{b}(y)\delta(y)}{C_{M}yL(\log y)} \leq \sup_{b \in [M,M+1]} \sup_{x} \frac{ \left|T_{b}(x)\delta(x)\right|}{C_{M}yL(\log y)}\\
& \leq \frac{C}{C_{M}yL(\log y)} \sup_{b \in [M,M+1]}\left\|T_{b}\right\|_{N,1} = O\left(\frac{1}{yL(\log y)}\right).
\end{align*}

\end{proof}

\section{The Ingham-Karamata theorem}
\label{ingham section abserrorterm}
We start by stating the Ingham-Karamata theorem. A real-valued 
function $\tau$ is called \emph{slowly decreasing} \cite{korevaarbook} if for each $\varepsilon > 0$ there is $\delta > 0$ such that
\[
 \liminf_{x\to\infty}\inf_{h\in[0,\delta]}(\tau(x+h) - \tau(x)) > - \varepsilon.
\]
\begin{theorem}
\label{Inghamth abserrorWI} Let $\tau \in L^{1}_{loc}[0,\infty)$ be slowly decreasing and have convergent Laplace transform 
$$
\mathcal{L}\{\tau;s\}=\int_{0}^{\infty}\tau(x)e^{-sx}\mathrm{d}x \quad \mbox{ for }\ \Re e \: s > 0.
$$
Suppose that $\mathcal{L}\{\tau;s\}$ has a continuous extension to the imaginary axis. Then,
$$
\tau(x)=o(1).
$$
\end{theorem}

We also have the ensuing result on the absence of remainders in the Ingham-Karamata theorem.

\begin{theorem}\label{th3 abserrorW-I} Let $\eta$ be a positive function and let $-1<\alpha < 0$. Suppose that every  slowly decreasing function $\tau\in L^{1}_{loc}[0,\infty)$, whose Laplace transform converges on $\Re e\:s>0$ and has an analytic continuation to the half-plane $\Re e \: s > \alpha$, satisfies 
$$\tau(x) = O(\eta(x)).$$
Then, we necessarily have 
$$
\eta(x)=\Omega(1).
$$
\end{theorem}

\begin{proof} We reduce the problem to Theorem \ref{th2 abserrorW-I}. So set $\rho(x)=\eta(\log x)$ and we are going to show that $\rho(x)=\Omega(1)$. Suppose that $S$ is non-decreasing on $[1,\infty)$ such that its Mellin transform $G(s)$ converges on $\Re e\: s>1$ and  
$$
G(s)-\frac{1}{s-1}
$$
analytically extends to $\Re e\: s>1+\alpha$. By the Wiener-Ikehara theorem $\tau(x)=e^{-x}S(e^{x})-1=o(1)$ and in particular it is slowly decreasing. Its Laplace transform
$$
\mathcal{L}\{\tau;s\}= G(s+1)-\frac{1}{s}
$$
is analytic on $\Re e \: s>\alpha$ and thus $\tau(x)=O(\eta(x))$, or equivalently, $S(x)=x+O(x\rho(x))$. Since $S$ was arbitrary, Theorem \ref{th2 abserrorW-I} gives at once $\rho(x)=\Omega(1)$. The proof is complete.
\end{proof}


\begin{thebibliography}{99}


\bibitem{aramaki} J. Aramaki, \emph{An extension of the Ikehara Tauberian theorem and its application,} Acta Math. Hungar. \textbf{71} (1996), 297--326.

\bibitem{B-B-T2016} C.~J.~K.~Batty, A.~Borichev, Y.~Tomilov, \emph{$L^p$-tauberian theorems and $L^p$-rates for energy decay,} J. Funct. Anal. \textbf{270} (2016), 1153--1201.

\bibitem{binghambook} N.~H.~Bingham, C.~M.~Goldie, J.~L.~Teugels, \textit{Regular variation,} Encyclopedia of Mathematics and its Applications,  27, Cambridge University Press, Cambridge, 1989.

\bibitem{Chill-Seifert2016} R.~Chill, D.~Seifert, \emph{Quantified versions of Ingham's theorem,} Bull. Lond. Math. Soc. \textbf{48} (2016), 519--532.

\bibitem{C-Qbook} D.~Choimet, H.~Queff\'{e}lec,  \emph{Twelve landmarks of twentieth-century analysis,} Cambridge University Press, New York, 2015.


\bibitem{d-vW-I2016}  G.~Debruyne, J.~Vindas, \emph{Generalization of the Wiener-Ikehara theorem,} Illinois J. Math. \textbf{60} (2016), 613--624.

\bibitem{d-vOptIngham} G.~Debruyne, J.~Vindas, \emph{Optimal Tauberian constant in Ingham's theorem for Laplace transforms,} Israel J. Math., in press, DOI 10.1007/s11856-018-1758-1.

\bibitem{d-vCT}G.~Debruyne, J.~Vindas, \emph{Complex Tauberian theorems for Laplace transforms with local pseudofunction boundary behavior,} J. Anal. Math., to appear (preprint: arXiv:1604.05069).

\bibitem{delange1954}H.~Delange, \emph{G\'{e}n\'{e}ralisation du th\'{e}or\`{e}me de Ikehara,} Ann. Sci. Ecole Norm. Sup. \textbf{71} (1954), 213--242. 

\bibitem{ganelius} T.~Ganelius, \emph{Tauberian remainder theorems,} Lecture Notes in Mathematics, 232, Springer-Verlag, Berlin-New York, 1971.

\bibitem{grahamvaaler} S.~W.~Graham, J.~D.~Vaaler, \emph{A class of extremal functions for the Fourier transform,} Trans. Amer. Math. Soc. \textbf{265} (1981), 283--302.


\bibitem{korevaarbook} J.~Korevaar, \emph{Tauberian theory. A century of developments}, Grundlehren der Mathematischen Wissenschaften, 329, Springer-Verlag, Berlin, 2004.

\bibitem{korevaar2006}  J.~Korevaar,  \emph{Distributional Wiener-Ikehara  theorem and twin  primes,} Indag. Math. (N.S.) \textbf{16} (2005), 37--49.

\bibitem{L-Zbook} P.~D.~Lax, L.~Zalcman, \emph{Complex proofs of real theorems,}
University Lecture Series, 58, American Mathematical Society, Providence, RI, 2012.

\bibitem{muger} M.~M\"uger, \emph{On Ikehara type Tauberian theorems with $O(x^{\gamma})$ remainders}, Abh. Math. Semin. Univ. Hambg., in press, doi:10.1007/s12188-017-0187-0.

\bibitem{revesz-roton} Sz.~Gy.~R\'{e}v\'{e}sz, A.~de~Roton, \emph{Generalization of the effective Wiener-Ikehara theorem,} Int. J. Number Theory \textbf{9} (2013), 2091--2128. 

\bibitem{Seifert2015} D.~Seifert, \emph{A quantified Tauberian theorem for sequences,} Studia Math. \textbf{227} (2015), 183--192.

\bibitem{Tenenbaumbook} G.~Tenenbaum, \emph{Introduction to analytic and probabilistic number theory,} American Mathematical Society, Providence, RI, 2015.

\bibitem {treves} F.~Tr\`{e}ves, \emph{Topological vector spaces,
distributions and kernels}, Academic Press, New York, 1967.
	
	
\bibitem{zhang2014} W.-B.~Zhang, \emph{Wiener-Ikehara theorems and the Beurling generalized primes,} Monatsh. Math. 174 (2014), 627--652.	


\end{thebibliography}
\end{document}